\documentclass[10pt]{article}
\usepackage{amsmath}
\usepackage{amssymb}
\usepackage{amsthm}
\usepackage{mathrsfs}
\numberwithin{equation}{section}

\begin{document}
\title{Some estimates for commutators of Calder\'on-Zygmund operators on weighted Morrey spaces}
\author{Hua Wang \footnote{E-mail address: wanghua@pku.edu.cn.}\\
\footnotesize{Department of Mathematics, Zhejiang University, Hangzhou 310027, China}}
\date{}
\maketitle

\begin{abstract}
Let $T$ be a Calder\'on-Zygmund singular integral operator. In this paper, we will use a unified approach to show some boundedness properties of commutators $[b,T]$ on the weighted Morrey spaces $L^{p,\kappa}(w)$ under appropriate conditions on the weight $w$, where the symbol $b$ belongs to weighted $BMO$ or Lipschitz space or weighted Lipschitz space.
\end{abstract}

\section{Introduction}

The classical Morrey spaces $\mathcal L^{p,\lambda}$ were originally introduced by Morrey in \cite{morrey} to study the local behavior of solutions to second order elliptic partial differential equations. For the properties and applications of classical Morrey spaces, we refer the readers to \cite{morrey,peetre}. In \cite{chiarenza}, Chiarenza and Frasca showed the boundedness of the Hardy-Littlewood maximal operator, the fractional integral operator and the Calder\'on-Zygmund singular integral operator on these spaces.

In 2009, Komori and Shirai \cite{komori} defined the weighted Morrey spaces $L^{p,\kappa}(w)$ and studied the boundedness of the above classical operators on these weighted spaces. Suppose that $T$ is a Calder\'on-Zygmund singular integral operator and $b$ is a locally integrable function on $\mathbb R^n$, the commutator generated by $b$ and $T$ is defined as follows
\begin{equation*}
[b,T]f(x)=b(x)Tf(x)-T(bf)(x).
\end{equation*}

In \cite{komori}, Komori and Shirai proved that when $b\in BMO(\mathbb R^n)$, $1<p<\infty$, $0<\kappa<1$ and $w\in A_p$(Muckenhoupt weight class), then $[b,T]$ is bounded on $L^{p,\kappa}(w)$.

The main purpose of this paper is to discuss the boundedness of commutators $[b,T]$ on the weighted Morrey spaces when the symbol $b$ belongs to some other function spaces. Our main results are stated as follows.

\newtheorem{theorem}{Theorem}[section]
\begin{theorem}
Let $1<p<\infty$, $0<\kappa<1$ and $w\in A_1$. Suppose that $b\in BMO(w)$$($weighted BMO$)$, then $[b,T]$ is bounded from $L^{p,\kappa}(w)$ to $L^{p,\kappa}(w^{1-p},w)$.
\end{theorem}

\begin{theorem}
Let $0<\beta<1$, $1<p<n/\beta$, $1/s=1/p-\beta/n$, $0<\kappa<\min\{p/s,{p\beta}/n\}$ and $w^s\in A_1$. Suppose that $b\in Lip_\beta(\mathbb R^n)$$($Lipschitz space$)$, then $[b,T]$ is bounded from $L^{p,\kappa}(w^p,w^s)$ to $L^{s,{\kappa s}/p}(w^s)$.
\end{theorem}

\begin{theorem}
Let $0<\beta<1$, $1<p<n/\beta$, $1/s=1/p-\beta/n$, $0<\kappa<p/s$ and $w^{s/p}\in A_1$. Suppose that $b\in Lip_\beta(w)$$($weighted Lipschitz space$)$ and $r_w>\frac{1-\kappa}{p/s-\kappa}$, then $[b,T]$ is bounded from $L^{p,\kappa}(w)$ to $L^{s,{\kappa s}/p}(w^{1-s},w)$, where $r_w$ denotes the critical index of $w$ for the reverse H\"older condition.
\end{theorem}

\section{Definitions and Notations}

First let us recall some standard definitions and notations of weight classes. A weight $w$ is a locally integrable function on $\mathbb R^n$ which takes values in $(0,\infty)$ almost everywhere, all cubes are assumed to have their sides parallel to the coordinate axes. Given a cube $Q$ and $\lambda>0$, $\lambda Q$ denotes the cube with the same center as $Q$ whose side length is $\lambda$ times that of $Q$, $Q=Q(x_0,r_Q)$ denotes the cube centered at $x_0$ with side length $r_Q$. For a given weight function $w$, we denote the Lebesgue measure of $Q$ by $|Q|$ and the weighted measure of $Q$ by $w(Q)$, where $w(Q)=\int_Qw(x)\,dx.$

\newtheorem{defn}[theorem]{Definition}
\begin{defn}[\cite{muckenhoupt1}]
A weight function $w$ is in the Muckenhoupt class $A_p$ with $1<p<\infty$ if for every cube $Q$ in $\mathbb R^n$, there exists a positive constant $C$ which is independent of $Q$ such that
$$\left(\frac1{|Q|}\int_Q w(x)\,dx\right)\left(\frac1{|Q|}\int_Q w(x)^{-1/{(p-1)}}\,dx\right)^{p-1}\le C.$$
When $p=1$, $w\in A_1$, if
$$\frac1{|Q|}\int_Q w(x)\,dx\le C\cdot\underset{x\in Q}{\mbox{ess\,inf}}\,w(x).$$
When $p=\infty$, we define $A_\infty=\bigcup_{1<p<\infty}A_p$.
\end{defn}

\begin{defn}[\cite{muckenhoupt2}]
A weight function $w$ belongs to $A_{p,q}$ for $1<p<q<\infty$ if for every cube $Q$ in $\mathbb R^n$, there exists a positive constant $C$ which is independent of $Q$ such that
$$\left(\frac{1}{|Q|}\int_Q w(x)^q\,dx\right)^{1/q}\left(\frac{1}{|Q|}\int_Q w(x)^{-p'}\,dx\right)^{1/{p'}}\le C,$$where $p'$ denotes the conjugate exponent of $p>1;$ that is, $1/p+1/{p'}=1$.
\end{defn}

\begin{defn}[\cite{garcia2}]
A weight function $w$ belongs to the reverse H\"older class $RH_r$ if there exist two constants $r>1$ and $C>0$ such that the following reverse H\"older inequality
$$\left(\frac{1}{|Q|}\int_Q w(x)^r\,dx\right)^{1/r}\le C\left(\frac{1}{|Q|}\int_Q w(x)\,dx\right)$$
holds for every cube $Q$ in $\mathbb R^n$.
\end{defn}

It is well known that if $w\in A_p$ with $1<p<\infty$, then $w\in A_r$ for all $r>p$, and $w\in A_q$ for some $1<q<p$. If $w\in A_p$ with $1\le p<\infty$, then there exists $r>1$ such that $w\in RH_r$. It follows directly from H\"older's inequality that $w\in RH_r$ implies $w\in RH_s$ for all $1<s<r$. Moreover, if $w\in RH_r$, $r>1$, then we have $w\in RH_{r+\varepsilon}$ for some $\varepsilon>0$. We thus write $r_w\equiv\sup\{r>1:w\in RH_r\}$ to denote the critical index of $w$ for the reverse H\"older condition.

We state the following results that we will use frequently in the sequel.

\newtheorem*{lemmaA}{Lemma A}
\begin{lemmaA}[\cite{garcia2}]
Let $w\in A_p$, $p\ge1$. Then, for any cube $Q$, there exists an absolute constant $C>0$ such that $$w(2Q)\le C\,w(Q).$$ In general, for any $\lambda>1$, we have $$w(\lambda Q)\le C\cdot\lambda^{np}w(Q),$$where $C$ does not depend on $Q$ nor on $\lambda$.
\end{lemmaA}

\newtheorem*{lemmab}{Lemma B}
\begin{lemmab}[\cite{garcia2,gundy}]
Let $w\in A_p\cap RH_r$, $p\ge1$ and $r>1$. Then there exist two constants $C_1$, $C_2>0$ such that
$$C_1\left(\frac{|E|}{|Q|}\right)^p\le\frac{w(E)}{w(Q)}\le C_2\left(\frac{|E|}{|Q|}\right)^{(r-1)/r}$$for any measurable subset $E$ of a cube $Q$.
\end{lemmab}

\newtheorem*{lemmac}{Lemma C}
\begin{lemmac}[\cite{johnson}]
Let $s>1$, $1\le p<\infty$ and $A^s_p=\big\{w:w^s\in A_p\big\}.$ Then
$$A^s_p=A_{1+(p-1)/s}\cap RH_s.$$ In particular, $$A^s_1=A_1\cap RH_s.$$
\end{lemmac}

Next we shall introduce the Hardy-Littlewood maximal operator and several variants, the Calder\'on-Zygmund operator and some function spaces.

\begin{defn}
The Hardy-Littlewood maximal operator $M$ is defined by
$$M(f)(x)=\sup_{x\in Q}\frac{1}{|Q|}\int_Q|f(y)|\,dy.$$
For $0<\beta<n$, $r\ge1$, we define the fractional maximal operator $M_{\beta,r}$ by
$$M_{\beta,r}(f)(x)=\sup_{x\in Q}\bigg(\frac{1}{|Q|^{1-\frac{\beta r}{n}}}\int_Q|f(y)|^r\,dy\bigg)^{1/r}.$$
Let $w$ be a weight. The weighted maximal operator $M_w$ is defined by
$$M_w(f)(x)=\underset{x\in Q}{\sup}\frac{1}{w(Q)}\int_Q|f(y)|w(y)\,dy.$$
For $0<\beta<n$ and $r\ge1$, we define the fractional weighted maximal operator $M_{\beta,r,w}$ by
$$M_{\beta,r,w}(f)(x)=\sup_{x\in Q}\bigg(\frac{1}{w(Q)^{1-\frac{\beta r}{n}}}\int_Q|f(y)|^rw(y)\,dy\bigg)^{1/r},$$
where the above supremum is taken over all cubes $Q$ containing $x$.
\end{defn}

\begin{defn}
We say that $T$ is a Calder\'on-Zygmund singular integral operator if there exists a kernel function $K$ which satisfies the following conditions
\end{defn}

(a) $Tf(x)=\mbox{\upshape{P.V.}}\int_{\mathbb R^n}K(x-y)f(y)\,dy$;

(b) $|K(x)|\le C|x|^{-n}\quad x\neq0$;

(c) $|K(x-y)-K(x)|\le C|y|/{|x|^{n+1}}\quad|x|\ge2|y|>0$.

Let $1\le p<\infty$ and $w$ be a weight function. A locally integrable function $b$ is said to be in $BMO_p(w)$ if
$$\|b\|_{BMO_{p}(w)}=\sup_{Q}\left(\frac{1}{w(Q)}\int_Q\big|b(x)-b_Q\big|^{p}w(x)^{1-p}\,dx\right)^{1/p}<\infty,$$
where $b_Q=\frac{1}{|Q|}\int_Q b(y)\,dy$ and the supremum is taken over all cubes $Q\subseteq{\mathbb R^n}$.

Let $0<\beta<1$ and $1\le p<\infty$. A locally integrable function $b$ is said to be in $Lip_\beta^p(\mathbb R^n)$ if
$$\|b\|_{Lip_\beta^p}=\sup_{Q}\frac{1}{|Q|^{\beta/n}}\left(\frac{1}{|Q|}
\int_Q\big|b(x)-b_Q\big|^p\,dx\right)^{1/p}<\infty.$$

Let $0<\beta<1$, $1\le p<\infty$ and $w$ be a weight function. A locally integrable function $b$ is said to belong to $Lip_{\beta}^p(w)$ if
$$\|b\|_{Lip_\beta^p(w)}=\sup_{Q}\frac{1}{w(Q)^{\beta/n}}\left(\frac{1}{w(Q)}
\int_Q\big|b(x)-b_Q\big|^{p}w(x)^{1-p}\,dx\right)^{1/p}<\infty.$$
Moreover, we denote simply by $BMO(w)$, $Lip_\beta(\mathbb R^n)$ and $Lip_\beta(w)$ when $p=1$.

\newtheorem*{lemmad}{Lemma D}
\begin{lemmad}[\cite{garcia1,pal}]
$(i)$ Let $w\in A_1$. Then for any $1\le p<\infty$, there exists an absolute constant $C>0$ such that $\|b\|_{BMO_{p}(w)}\le C\|b\|_{BMO(w)}.$\\
$(ii)$ Let $0<\beta<1$. Then for any $1\le p<\infty$, there exists an absolute constant $C>0$ such that $\|b\|_{Lip_\beta^p}\le C\|b\|_{Lip_\beta}.$\\
$(iii)$ Let $0<\beta<1$ and $w\in A_1$. Then for any $1\le p<\infty$, there exists an absolute constant $C>0$ such that $\|b\|_{Lip_\beta^p(w)}\le C\|b\|_{Lip_\beta(w)}.$
\end{lemmad}

We are going to conclude this section by defining the weighted Morrey space and giving some known results relevant to this paper. We refer the readers to \cite{komori} for further details.

\begin{defn}[\cite{komori}]
Let $1\le p<\infty$, $0<\kappa<1$ and $w$ be a weight function. Then the weighted Morrey space is defined by
$$L^{p,\kappa}(w)=\big\{f\in L^p_{loc}(w):\|f\|_{L^{p,\kappa}(w)}<\infty\big\},$$
where
$$\|f\|_{L^{p,\kappa}(w)}=\underset{Q}{\sup}\left(\frac{1}{w(Q)^\kappa}\int_Q|f(x)|^pw(x)\,dx\right)^{1/p}$$
and the supremum is taken over all cubes $Q$ in $\mathbb R^n$.
\end{defn}

\newtheorem{rem}[theorem]{Remark}
\begin{rem}
Equivalently, we could define the weighted Morrey space with balls instead of cubes. Hence we shall use these two definitions of weighted Morrey space appropriate to calculations.
\end{rem}

In order to deal with the fractional order case, we need to consider the weighted Morrey space with two weights.

\begin{defn}[\cite{komori}]
Let $1\le p<\infty$ and $0<\kappa<1$. Then for two weights $u$ and $v$, the weighted Morrey space is defined by
$$L^{p,\kappa}(u,v)=\big\{f\in L^p_{loc}(u):\|f\|_{L^{p,\kappa}(u,v)}<\infty\big\},$$
where
$$\|f\|_{L^{p,\kappa}(u,v)}=\sup_{Q}\left(\frac{1}{v(Q)^{\kappa}}\int_Q|f(x)|^pu(x)\,dx\right)^{1/p}.$$
\end{defn}

\newtheorem*{thmE}{Theorem E}
\begin{thmE}[\cite{komori}]
If $1<p<\infty$, $0<\kappa<1$ and $w\in A_\infty$, then $M_w$ is bounded on $L^{p,\kappa}(w)$.
\end{thmE}

\newtheorem*{thmF}{Theorem F}
\begin{thmF}[\cite{komori}]
If $1<p<\infty$, $0<\kappa<1$ and $w\in A_p$, then $M$ is bounded on $L^{p,\kappa}(w)$.
\end{thmF}

\newtheorem*{thmG}{Theorem G}
\begin{thmG}[\cite{komori}]
If $1<p<\infty$, $0<\kappa<1$ and $w\in A_p$, then $T$ is bounded on $L^{p,\kappa}(w)$.
\end{thmG}

\newtheorem*{thmH}{Theorem H}
\begin{thmH}[\cite{komori}]
If $0<\beta<n$, $1<p<n/{\beta}$, $1/s=1/p-{\beta}/n$, $0<\kappa<p/s$ and $w\in A_{p,s}$, then $M_{\beta,1}$ is bounded from $L^{p,\kappa}(w^p,w^s)$ to $L^{s,{\kappa s}/p}(w^s)$.
\end{thmH}

Throughout this article, we will use $C$ to denote a positive constant, which is independent of the main parameters and not necessarily the same at each occurrence. By $A\sim B$, we mean that there exists a constant $C>1$ such that $\frac1C\le\frac AB\le C$. Moreover, we will denote the conjugate exponent of $r>1$ by $r'=r/(r-1).$

\section{Proof of Theorem 1.1}

We shall adopt a unified approach(sharp maximal function estimate) to deal with all the cases. Following the idea given in \cite{perez}, for $0<\delta<1$, we define the $\delta$-sharp maximal operator as
$M^{\#}_\delta(f)=M^{\#}(|f|^\delta)^{1/{\delta}},$
which is a modification of the sharp maximal operator $M^{\#}$ of Fefferman and Stein \cite{stein2}. We also set $M_\delta(f)=M(|f|^\delta)^{1/{\delta}}$. Suppose that $w\in A_\infty$, then for any cube $Q$, we have the following weighted version of the local good-$\lambda$ inequality(see \cite{stein2})
\begin{equation*}
w\Big(\Big\{x\in Q:M_{\delta}f(x)>\lambda,M^{\#}_\delta f(x)\le \lambda\varepsilon\Big\}\Big)\le C\varepsilon\cdot w\Big(\Big\{x\in Q:M_{\delta}f(x)>\frac{\lambda}{2}\Big\}\Big),
\end{equation*}
for all $\lambda,\varepsilon>0$. As a consequence, by using the standard arguments(see \cite{stein2,torchinsky}), we can establish the following estimate, which will play a key role in the proof of our main results.

\newtheorem{prop}[theorem]{Proposition}
\begin{prop}
Let $0<\delta<1$, $1<p<\infty$ and $0<\kappa<1$. If $u,v\in A_\infty$, then we have
\begin{equation*}
\big\|M_\delta(f)\big\|_{L^{p,\kappa}(u,v)}\le C\big\|M^{\#}_\delta(f)\big\|_{L^{p,\kappa}(u,v)}
\end{equation*}
for all functions $f$ such that the left hand side is finite. In particular, when $u=v=w$ and $w\in A_\infty$, then we have
\begin{equation*}
\big\|M_\delta(f)\big\|_{L^{p,\kappa}(w)}\le C\big\|M^{\#}_\delta(f)\big\|_{L^{p,\kappa}(w)}
\end{equation*}
for all functions $f$ such that the left hand side is finite.
\end{prop}

In order to simplify the notation, we set $M_{0,r,w}=M_{r,w}$. Then we will prove the following lemma.

\newtheorem{lemma}[theorem]{Lemma}
\begin{lemma}
Let $1<p<\infty$, $0<\kappa<1$ and $w\in A_\infty$. Then for any $1<r<p$, we have
\begin{equation*}
\|M_{r,w}(f)\|_{L^{p,\kappa}(w)}\le C\|f\|_{L^{p,\kappa}(w)}.
\end{equation*}
\end{lemma}

\begin{proof}
With the notations mentioned earlier, we know that
$$M_{r,w}(f)=M_w(|f|^r)^{1/r}.$$
From the definition, we readily see that
$$\|M_{r,w}(f)\|_{L^{p,\kappa}(w)}=\big\|M_w(|f|^r)\big\|_{L^{p/r,\kappa}(w)}^{1/r}.$$
Since $1<r<p$, then $p/r>1$. Hence, by using Theorem E, we obtain
$$\big\|M_w(|f|^r)\big\|_{L^{p/r,\kappa}(w)}^{1/r}\le C\big\||f|^r\big\|_{L^{p/r,\kappa}(w)}^{1/r}\le C\|f\|_{L^{p,\kappa}(w)}.$$
We are done.
\end{proof}

\begin{prop}
Let $0<\delta<1$, $w\in A_1$ and $b\in BMO(w)$. Then for all $r>1$ and for all $x\in \mathbb R^n$, we have
\begin{equation*}
\begin{split}
M^{\#}_\delta([b,T]f)(x)\le& C\|b\|_{BMO(w)}\Big(w(x)M_{r,w}(Tf)(x)+w(x)M_{r,w}(f)(x)\\
&+w(x)M(f)(x)\Big).
\end{split}
\end{equation*}
\end{prop}

\begin{proof}
For any given $x\in\mathbb R^n$, fix a ball $B=B(x_0,r_B)$ which contains $x$, where $B(x_0,r_B)$
denotes the ball with the center $x_0$ and radius $r_B$. We decompose $f=f_1+f_2$, where $f_1=f\chi_{_{2B}}$, $\chi_{_{2B}}$ denotes the characteristic function of $2B$. Observe that
\begin{equation*}
[b,T]f(x)=(b(x)-b_{2B})Tf(x)-T((b-b_{2B})f)(x).
\end{equation*}
Since $0<\delta<1$, then for arbitrary constant $c$, we have
\begin{align}
&\bigg(\frac{1}{|B|}\int_B\Big||[b,T]f(y)|^{\delta}-|c|^\delta\Big|\,dy\bigg)^{1/\delta}\\
\le&\left(\frac{1}{|B|}\int_B\big|[b,T]f(y)-c\big|^{\delta}\,dy\right)^{1/\delta}\notag\\
\le&C\left(\frac{1}{|B|}\int_B\big|(b(y)-b_{2B})Tf(y)\big|^{\delta}\,dy\right)^{1/\delta}\notag\\
&+C\left(\frac{1}{|B|}\int_B\big|T((b-b_{2B})f_1)(y)\big|^\delta\,dy\right)^{1/{\delta}}\notag\\
&+C\left(\frac{1}{|B|}\int_B\big|T((b-b_{2B})f_2)(y)+c\big|^\delta\,dy\right)^{1/\delta}\notag\\
=&\;\mbox{\upshape I+II+III}\notag.
\end{align}
We are now going to estimate each term separately. Since $w\in A_1$, then it follows from H\"older's inequality and Lemma D that
\begin{align}
\mbox{\upshape I}&\le C\cdot\frac{1}{|B|}\int_B\big|(b(y)-b_{2B})Tf(y)\big|\,dy\notag\\
&\le C\cdot\frac{1}{|B|}\bigg(\int_B\big|b(y)-b_{2B}\big|^{r'}w^{1-r'}\,dy\bigg)^{1/{r'}}
\bigg(\int_B\big|Tf(y)\big|^rw(y)\,dy\bigg)^{1/r}\notag\\
&\le C\|b\|_{BMO(w)}\cdot\frac{w(B)}{|B|}\bigg(\frac{1}{w(B)}\int_B\big|Tf(y)\big|^rw(y)\,dy\bigg)^{1/r}\notag\\
&\le C\|b\|_{BMO(w)}w(x)M_{r,w}(Tf)(x).
\end{align}
Applying Kolmogorov's inequality(see [3, p.485]), H\"older's inequality and Lemma D, we can get
\begin{align}
\mbox{\upshape II}&\le C\cdot\frac{1}{|B|}\int_{2B}\big|(b(y)-b_{2B})f(y)\big|\,dy\notag\\
&\le C\cdot\frac{1}{|B|}\bigg(\int_{2B}\big|b(y)-b_{2B}\big|^{r'}w^{1-r'}\,dy\bigg)^{1/{r'}}
\bigg(\int_{2B}\big|f(y)\big|^rw(y)\,dy\bigg)^{1/r}\notag\\
&\le C\|b\|_{BMO(w)}\cdot\frac{w(2B)}{|2B|}\bigg(\frac{1}{w(2B)}\int_{2B}\big|f(y)\big|^rw(y)\,dy\bigg)^{1/r}\notag\\
&\le C\|b\|_{BMO(w)}w(x)M_{r,w}(f)(x).
\end{align}
To estimate the last term III, we first fix the value of $c$ by taking $c=-T((b-b_{2B})f_2)(x_0)$, then we obtain
\begin{equation*}
\begin{split}
\mbox{\upshape III}\le&\,C\cdot\frac{1}{|B|}\int_B\big|T((b-b_{2B})f_2)(y)-T((b-b_{2B})f_2)(x_0)\big|\,dy\\
\le&\, C\cdot\frac{1}{|B|}\int_B\int_{(2B)^c}|K(y,z)-K(x_0,z)||b(z)-b_{2B}||f(z)|\,dzdy\\
\le&\, C\cdot\frac{1}{|B|}\int_B\bigg(\sum_{j=1}^{\infty}\int_{2^{j+1}B\backslash2^j B}\frac{|y-x_0|}{|z-x_0|^{n+1}}|b(z)-b_{2B}||f(z)|\,dz\bigg)dy\\
\le&\, C\sum_{j=1}^\infty\frac{1}{2^j}\frac{1}{|2^{j+1}B|}\int_{2^{j+1}B}|b(z)-b_{2^{j+1}B}||f(z)|\,dz\\
&+C\sum_{j=1}^\infty\frac{1}{2^j}\frac{1}{|2^{j+1}B|}\int_{2^{j+1}B}|b_{2^{j+1}B}-b_{2B}||f(z)|\,dz\\
=&\,\mbox{\upshape{IV+V}}.
\end{split}
\end{equation*}
As in the estimate of II, we can also get
\begin{align}
\mbox{\upshape{IV}}&\le C\|b\|_{BMO(w)}\sum_{j=1}^\infty\frac{1}{2^j}\cdot w(x)M_{r,w}(f)(x)\notag\\
&\le C\|b\|_{BMO(w)}w(x)M_{r,w}(f)(x).
\end{align}
Note that $w\in A_1$, a direct calculation shows that
\begin{equation}
|b_{2^{j+1}B}-b_{2B}|\le C\|b\|_{BMO(w)}j\cdot w(x).
\end{equation}
Substituting the above inequality (3.5) into the term V, we thus obtain
\begin{equation}
\mbox{\upshape{V}}\le C\|b\|_{BMO(w)}\sum_{j=1}^\infty\frac{j}{2^j}\cdot w(x)M(f)(x)\le C\|b\|_{BMO(w)}w(x)M(f)(x).
\end{equation}
Combining the above estimates (3.2)--(3.4) with (3.6) and taking the supremum over all balls $B\subseteq\mathbb R^n$, we get the desired result.
\end{proof}

We are now in a position to give the proof of Theorem 1.1.

\begin{proof}[Proof of Theorem $1.1$]
For any $1<p<\infty$, we can choose a positive number $r$ such that $1<r<p$. Applying Proposition 3.1 and Proposition 3.3, we thus have
\begin{equation*}
\begin{split}
&\big\|[b,T]f\big\|_{L^{p,\kappa}(w^{1-p},w)}\\
\le&\, C\big\|M^{\#}_\delta([b,T]f)\big\|_{L^{p,\kappa}(w^{1-p},w)}\\
\le&\, C\|b\|_{BMO(w)}\Big(\|w(\cdot)M_{r,w}(Tf)\|_{L^{p,\kappa}(w^{1-p},w)}+
\|w(\cdot)M_{r,w}(f)\|_{L^{p,\kappa}(w^{1-p},w)}\\
&+\|w(\cdot)M(f)\|_{L^{p,\kappa}(w^{1-p},w)}\Big)\\
\le&\, C\|b\|_{BMO(w)}\Big(\|M_{r,w}(Tf)\|_{L^{p,\kappa}(w)}+\|M_{r,w}(f)\|_{L^{p,\kappa}(w)}+\|M(f)\|_{L^{p,\kappa}(w)}\Big).
\end{split}
\end{equation*}
Therefore, by using Theorem F, Theorem G and Lemma 3.2, we obtain
\begin{equation*}
\begin{split}
\big\|[b,T]f\big\|_{L^{p,\kappa}(w^{1-p},w)}&\le C\|b\|_{BMO(w)}\Big(\|Tf\|_{L^{p,\kappa}(w)}+\|f\|_{L^{p,\kappa}(w)}\Big)\\
&\le C\|b\|_{BMO(w)}\|f\|_{L^{p,\kappa}(w)}.
\end{split}
\end{equation*}
This completes the proof of Theorem 1.1.
\end{proof}

\section{Proof of Theorem 1.2}

We begin with some lemmas which will be used in the proof of Theorem 1.2.

\begin{lemma}
Let $0<\beta<n$, $1<p<n/\beta$, $1/s=1/p-\beta/n$ and $w^s\in A_1$. Then for every $0<\kappa<p/s$ and $1<r<p$, we have
\begin{equation*}
\|M_{\beta,r}(f)\|_{L^{s,{\kappa s}/p}(w^s)}\le C\|f\|_{L^{p,\kappa}(w^p,w^s)}.
\end{equation*}
\end{lemma}

\begin{proof}
Note that
\begin{equation*}
M_{\beta,r}(f)=M_{\beta r,1}(|f|^r)^{1/r}.
\end{equation*}
From the definition, we can easily check that
\begin{equation}
w\in A_{p,s}\quad \mbox{if and only if}\quad w^s\in A_{1+s/{p'}}.
\end{equation}
Since $w^s\in A_1$, then we have $(w^r)^{s/r}\in A_{1+(s/r)/{(p/r)}'}$, which implies $w^r\in A_{p/r,s/r}$. Observe that $r/s=r/p-{\beta r}/n$. Then by Theorem H, we obtain that the fractional maximal operator $M_{\beta r,1}$ is bounded from $L^{p/r,\kappa}(w^p,w^s)$ to $L^{s/r,{\kappa s}/p}(w^s)$. Consequently
\begin{equation*}
\begin{split}
\|M_{\beta,r}(f)\|_{L^{s,{\kappa s}/p}(w^s)}&=\big\|M_{\beta r,1}(|f|^r)\big\|^{1/r}_{L^{s/r,{\kappa s}/p}(w^s)}\\
&\le C\big\||f|^r\big\|^{1/r}_{L^{p/r,\kappa}(w^p,w^s)}\\
&\le C\|f\|_{L^{p,\kappa}(w^p,w^s)}.
\end{split}
\end{equation*}
We are done.
\end{proof}

\begin{lemma}
Let $0<\beta<n$, $1<p<n/\beta$, $1/s=1/p-\beta/n$ and $w^s\in A_1$. Then for every $0<\kappa<{\beta p}/n$, we have
\begin{equation*}
\|T(f)\|_{L^{p,\kappa}(w^p,w^s)}\le C\|f\|_{L^{p,\kappa}(w^p,w^s)}.
\end{equation*}
\end{lemma}

\begin{proof}
Fix a ball $B=B(x_0,r_B)\subseteq\mathbb R^n$ and decompose $f=f_1+f_2$, where $f_1=f\chi_{_{2B}}$. Then we have
\begin{equation*}
\begin{split}
&\frac{1}{w^s(B)^{\kappa/p}}\bigg(\int_B|Tf(x)|^pw(x)^p\,dx\bigg)^{1/p}\\
\le&\,\frac{1}{w^s(B)^{\kappa/p}}\bigg(\int_B|Tf_1(x)|^pw(x)^p\,dx\bigg)^{1/p}
+\frac{1}{w^s(B)^{\kappa/p}}\bigg(\int_B|Tf_2(x)|^pw(x)^p\,dx\bigg)^{1/p}\\
=&\,J_1+J_2.
\end{split}
\end{equation*}
Since $w^s\in A_1$, $1<p<s$, then $w^p\in A_1$, which implies $w^p\in A_p$. The $L^p_w$ boundedness of $T$ and Lemma A yield
\begin{align}
J_1&\le C\cdot\frac{1}{w^s(B)^{\kappa/p}}\bigg(\int_{2B}|f(x)|^pw(x)^p\,dx\bigg)^{1/p}\notag\\
&\le C\|f\|_{L^{p,\kappa}(w^p,w^s)}\cdot\frac{w^s(2B)^{\kappa/p}}{w^s(B)^{\kappa/p}}\notag\\
&\le C\|f\|_{L^{p,\kappa}(w^p,w^s)}.
\end{align}
We now turn to estimate the term $J_2$. Note that when $x\in B$, $y\in(2B)^c$, then $|y-x|\sim|y-x_0|$. It follows from H\"older's inequality and the $A_p$ condition that
\begin{equation*}
\begin{split}
|T(f_2)(x)|\le&\, C\int_{(2B)^c}\frac{|f(y)|}{|x-y|^n}\,dy\\
\le&\, C\sum_{j=1}^\infty\frac{1}{|2^{j+1}B|}\int_{2^{j+1}B}|f(y)|\,dy\\
\le& \,C\sum_{j=1}^\infty\frac{1}{|2^{j+1}B|}\cdot|2^{j+1}B|w^p(2^{j+1}B)^{-1/p}
\bigg(\int_{2^{j+1}B}|f(y)|^pw(y)^p\,dy\bigg)^{1/p}\\ \le& \,C\|f\|_{L^{p,\kappa}(w^p,w^s)}\sum_{j=1}^\infty\frac{w^s(2^{j+1}B)^{\kappa/p}}{w^p(2^{j+1}B)^{1/p}}.
\end{split}
\end{equation*}
Hence
\begin{equation*}
\begin{split}
J_2&\le C\|f\|_{L^{p,\kappa}(w^p,w^s)}\sum_{j=1}^\infty\frac{w^p(B)^{1/p}}{w^p(2^{j+1}B)^{1/p}}
\cdot\frac{w^s(2^{j+1}B)^{\kappa/p}}{w^s(B)^{\kappa/p}}.
\end{split}
\end{equation*}
Since $w^s\in A_1$, then by Lemma B, we can get
$$C\cdot\frac{|B|}{|2^{j+1}B|}\le\frac{w^s(B)}{w^s(2^{j+1}B)}.$$
Since $s/p>1$ and $(w^p)^{s/p}\in A_1$, then by Lemma C, we have $w^p\in RH_{s/p}$. Hence, by using Lemma B again, we obtain
$$\frac{w^p(B)}{w^p(2^{j+1}B)}\le C\bigg(\frac{|B|}{|2^{j+1}B|}\bigg)^{1-p/s}.$$
Therefore
\begin{align}
J_2&\le C\|f\|_{L^{p,\kappa}(w^p,w^s)}\sum_{j=1}^\infty\big(2^{jn}\big)^{\kappa/p-\beta/n}\notag\\
&\le C\|f\|_{L^{p,\kappa}(w^p,w^s)},
\end{align}
where in the last inequality we have used the fact that $\kappa<{\beta p}/n$. Combining the above estimate (4.3) with (4.2) and taking the supremum over all balls $B\subseteq\mathbb R^n$, we conclude the proof of Lemma 4.2.
\end{proof}

\begin{prop}
Let $0<\delta<1$, $w\in A_1$, $0<\beta<1$ and $b\in Lip_{\beta}(\mathbb R^n)$. Then for all $r>1$ and for all $x\in \mathbb R^n$, we have
\begin{equation*}
M^{\#}_\delta([b,T]f)(x)\le C\|b\|_{Lip_\beta}\Big(M_{\beta,r}(Tf)(x)+M_{\beta,r}(f)(x)+M_{\beta,1}(f)(x)\Big).
\end{equation*}
\end{prop}

\begin{proof}
As in the proof of Proposition 3.3, we can split the previous expression (3.1) into three parts and estimate each term respectively. First, it follows from H\"older's inequality and Lemma D that
\begin{align}
\mbox{\upshape I}&\le C\cdot\frac{1}{|B|}\int_B\big|(b(y)-b_{2B})Tf(y)\big|\,dy\notag\\
&\le C\cdot\frac{1}{|B|}\bigg(\int_B\big|b(y)-b_{2B}\big|^{r'}\,dy\bigg)^{1/{r'}}
\bigg(\int_B\big|Tf(y)\big|^r\,dy\bigg)^{1/r}\notag\\
&\le C\|b\|_{Lip_\beta}\bigg(\frac{1}{|B|^{1-{\beta r}/n}}\int_B\big|Tf(y)\big|^r\,dy\bigg)^{1/r}\notag\\
&\le C\|b\|_{Lip_\beta}M_{\beta,r}(Tf)(x).
\end{align}
Applying Kolmogorov's inequality, H\"older's inequality and Lemma D, we can get
\begin{align}
\mbox{\upshape II}&\le C\cdot\frac{1}{|B|}\int_{2B}\big|(b(y)-b_{2B})f(y)\big|\,dy\notag\\
&\le C\|b\|_{Lip_\beta}M_{\beta,r}(f)(x).
\end{align}
Using the same arguments as that of Proposition 3.3, we have
$$\mbox{\upshape III}\le \mbox{IV+V},$$
where
$$\mbox{\upshape{IV}}=C\sum_{j=1}^\infty\frac{1}{2^j}\frac{1}{|2^{j+1}B|}\int_{2^{j+1}B}|b(z)-b_{2^{j+1}B}||f(z)|\,dz$$
and
$$\mbox{\upshape{V}}=C\sum_{j=1}^\infty\frac{1}{2^j}\frac{1}{|2^{j+1}B|}\int_{2^{j+1}B}|b_{2^{j+1}B}-b_{2B}||f(z)|\,dz.$$
As in the estimate of II, we can also deduce
\begin{equation}
\mbox{\upshape{IV}}\le C\|b\|_{Lip_\beta}M_{\beta,r}(f)(x)\sum_{j=1}^\infty\frac{1}{2^j}\le C\|b\|_{Lip_\beta}M_{\beta,r}(f)(x).
\end{equation}
By Lemma D, it is easy to check that
$$|b_{2^{j+1}B}-b_{2B}|\le C\|b\|_{Lip_\beta}\cdot j|2^{j+1}B|^{\beta/n}.$$
Hence
\begin{align}
\mbox{\upshape{V}}&\le C\|b\|_{Lip_\beta}\sum_{j=1}^\infty\frac{j}{2^j}\cdot\frac{1}{|2^{j+1}B|^{1-\beta/n}}\int_{2^{j+1}B}|f(z)|\,dz\notag\\ &\le C\|b\|_{Lip_\beta}M_{\beta,1}(f)(x)\sum_{j=1}^\infty\frac{j}{2^j}\notag\\
&\le C\|b\|_{Lip_\beta}M_{\beta,1}(f)(x).
\end{align}
Summarizing the estimates (4.4)--(4.7) derived above and taking the supremum over all balls $B\subseteq\mathbb R^n$, we obtain the desired result.
\end{proof}

Now we are able to prove our main result in this section.

\begin{proof}[Proof of Theorem $1.2$]
For $0<\beta<1$ and $1<p<n/\beta$, we can find a number $r$ such that $1<r<p$. Applying Proposition 3.1 and Proposition 4.3, we can get
\begin{equation*}
\begin{split}
\big\|[b,T]f\big\|_{L^{s,{\kappa s}/p}(w^s)}\le&\, C\big\|M^{\#}_\delta([b,T]f)\big\|_{L^{s,{\kappa s}/p}(w^s)}\\
\le&\, C\|b\|_{Lip_\beta}\Big(\|M_{\beta,r}(Tf)\|_{L^{s,{\kappa s}/p}(w^s)}+\|M_{\beta,r}(f)\|_{L^{s,{\kappa s}/p}(w^s)}\\
&+\|M_{\beta,1}(f)\|_{L^{s,{\kappa s}/p}(w^s)}\Big).
\end{split}
\end{equation*}
Since $w^s\in A_1$, then by (4.1), we have $w\in A_{p,s}$. Since $0<\kappa<\min\{p/s,{p\beta}/n\}$, by Theorem H, Lemma 4.1 and Lemma 4.2, we thus obtain
\begin{equation*}
\begin{split}
\big\|[b,T]f\big\|_{L^{s,{\kappa s}/p}(w^s)}&\le C\|b\|_{Lip_\beta}\Big(\|Tf\|_{L^{p,\kappa}(w^p,w^s)}+\|f\|_{L^{p,\kappa}(w^p,w^s)}\Big)\\
&\le C\|b\|_{Lip_\beta}\|f\|_{L^{p,\kappa}(w^p,w^s)}.
\end{split}
\end{equation*}
This completes the proof of Theorem 1.2.
\end{proof}

\section{Proof of Theorem 1.3}
Before proving our main theorem, we need to establish the following lemmas.

\begin{lemma}
Let $0<\beta<n$, $1<p<n/\beta$, $1/s=1/p-\beta/n$ and $w\in A_\infty$. Then for every $0<\kappa<p/s$, we have
\begin{equation*}
\|M_{\beta,1,w}(f)\|_{L^{s,{\kappa s}/p}(w)}\le C\|f\|_{L^{p,\kappa}(w)}.
\end{equation*}
\end{lemma}

\begin{proof}
Fix a cube $Q\subseteq\mathbb R^n$. Let $f=f_1+f_2$, where $f_1=f\chi_{_{2Q}}$. Since $M_{\beta,1,w}$ is a sublinear operator, then we have
\begin{equation*}
\begin{split}
&\frac{1}{w(Q)^{\kappa/p}}\bigg(\int_Q M_{\beta,1,w}f(x)^sw(x)\,dx\bigg)^{1/s}\\
\le&\,\frac{1}{w(Q)^{\kappa/p}}\bigg(\int_Q M_{\beta,1,w}f_1(x)^sw(x)\,dx\bigg)^{1/s}
+\frac{1}{w(Q)^{\kappa/p}}\bigg(\int_Q M_{\beta,1,w}f_2(x)^sw(x)\,dx\bigg)^{1/s}\\
=&\,K_1+K_2.
\end{split}
\end{equation*}
As we know, the fractional weighted maximal operator $M_{\beta,1,w}$ is bounded from $L^p(w)$ to $L^s(w)$. This together with Lemma A implies
\begin{align}
K_1&\le C\cdot\frac{1}{w(Q)^{\kappa/p}}\bigg(\int_{2Q}|f(x)|^pw(x)\,dx\bigg)^{1/p}\notag\\
&\le C\|f\|_{L^{p,\kappa}(w)}\cdot\frac{w(2Q)^{\kappa/p}}{w(Q)^{\kappa/p}}\notag\\
&\le C\|f\|_{L^{p,\kappa}(w)}.
\end{align}
We turn to deal with the term $K_2$. A simple geometric observation shows that for any $x\in Q$, we have
\begin{equation*}
M_{\beta,1,w}(f_2)(x)\le\sup_{R:\,Q\subseteq 3R}\frac{1}{w(R)^{1-\beta/n}}\int_R|f(y)|w(y)\,dy.
\end{equation*}
Since $0<\kappa<p/s$, then $(\kappa-1)/p+\beta/n<0$. By using H\"older's inequality and Lemma A, we can get
\begin{equation*}
\begin{split}
&\frac{1}{w(R)^{1-\beta/n}}\int_R|f(y)|w(y)\,dy\\
\le\,&\frac{1}{w(R)^{1-\beta/n}}\bigg(\int_R|f(y)|^pw(y)\,dy\bigg)^{1/p}\bigg(\int_Rw(y)\,dy\bigg)^{1/{p'}}\\
\le\,&C\|f\|_{L^{p,\kappa}(w)}w(R)^{(\kappa-1)/p+\beta/n}\\
\le\,&C\|f\|_{L^{p,\kappa}(w)}w(Q)^{(\kappa-1)/p+\beta/n}.
\end{split}
\end{equation*}
Hence
\begin{equation}
K_2\le C\|f\|_{L^{p,\kappa}(w)}w(Q)^{(\kappa-1)/p+\beta/n}w(Q)^{1/s}w(Q)^{-\kappa/p}\le C\|f\|_{L^{p,\kappa}(w)}.
\end{equation}
Combining the above inequality (5.2) with (5.1) and taking the supremum over all cubes $Q\subseteq\mathbb R^n$, we obtain the desired result.
\end{proof}

\begin{lemma}
Let $0<\beta<n$, $1<p<n/\beta$, $1/s=1/p-\beta/n$, $0<\kappa<p/s$ and $w\in A_\infty$. Then for any $1<r<p$, we have
\begin{equation*}
\|M_{\beta,r,w}(f)\|_{L^{s,{\kappa s}/p}(w)}\le C\|f\|_{L^{p,\kappa}(w)}.
\end{equation*}
\end{lemma}
\begin{proof}
Note that
\begin{equation*}
M_{\beta,r,w}(f)=M_{\beta r,1,w}(|f|^r)^{1/r}.
\end{equation*}
Since $1/s=1/p-\beta/n$, then for any $1<r<p$, we have $r/s=r/p-{\beta r}/n$.
Hence, by using Lemma 5.1, we obtain
\begin{equation*}
\begin{split}
\|M_{\beta,r,w}(f)\|_{L^{s,{\kappa s}/p}(w)}&=\big\|M_{\beta r,1,w}(|f|^r)\big\|^{1/r}_{L^{s/r,\kappa s/p}(w)}\\
&\le C\big\||f|^r\big\|^{1/r}_{L^{p/r,\kappa}(w)}\\
&\le C\|f\|_{L^{p,\kappa}(w)}.
\end{split}
\end{equation*}
We are done.
\end{proof}

\begin{lemma}
Let $0<\beta<n$, $1<p<n/\beta$, $1/s=1/p-\beta/n$ and $w^{s/p}\in A_1$. Then if $0<\kappa<p/s$ and $r_w>\frac{1-\kappa}{p/s-\kappa}$, we have
\begin{equation*}
\|M_{\beta,1}(f)\|_{L^{s,{\kappa s}/p}(w^{s/p},w)}\le C\|f\|_{L^{p,\kappa}(w)}.
\end{equation*}
\end{lemma}
\begin{proof}
Fix a ball $B=B(x_0,r_B)\subseteq\mathbb R^n$. Let $f=f_1+f_2$, where $f_1=f\chi_{_{2B}}$. Since $M_{\beta,1}$ is a sublinear operator, then we have
\begin{equation*}
\begin{split}
&\frac{1}{w(B)^{\kappa/p}}\bigg(\int_BM_{\beta,1}f(x)^sw(x)^{s/p}\,dx\bigg)^{1/s}\\
\le&\,\frac{1}{w(B)^{\kappa/p}}\bigg(\int_BM_{\beta,1}f_1(x)^sw(x)^{s/p}\,dx\bigg)^{1/s}
+\frac{1}{w(B)^{\kappa/p}}\bigg(\int_BM_{\beta,1}f_2(x)^sw(x)^{s/p}\,dx\bigg)^{1/s}\\
=&\,K_3+K_4.
\end{split}
\end{equation*}
For any function $f$, it is easy to see that
\begin{equation}
M_{\beta,1}(f)(x)\le C\cdot I_\beta(|f|)(x),
\end{equation}
where $I_\beta$ denotes the fractional integral operator(see \cite{stein1})
$$I_\beta(f)(x)=\frac{\Gamma(\frac{n-\beta}{2})}{2^\beta\pi^{\frac n2}\Gamma(\frac{\beta}{2})}\int_{\mathbb R^n}\frac{f(y)}{|x-y|^{n-\beta}}\,dy.$$
Since $w^{s/p}\in A_1$, then by (4.1), we have $w^{1/p}\in A_{p,s}$. It is well known that the fractional integral operator $I_\beta$ is bounded from $L^p(w^p)$ to $L^s(w^s)$ whenever $w\in A_{p,s}$(see \cite{muckenhoupt2}). This together with Lemma A gives
\begin{align}
K_3&\le C\cdot\frac{1}{w(B)^{\kappa/p}}\bigg(\int_{2B}|f(x)|^pw(x)\,dx\bigg)^{1/p}\notag\\
&\le C\|f\|_{L^{p,\kappa}(w)}\cdot\frac{w(2B)^{\kappa/p}}{w(B)^{\kappa/p}}\notag\\
&\le C\|f\|_{L^{p,\kappa}(w)}.
\end{align}
To estimate $K_4$, we note that when $x\in B$, $y\in(2B)^c$, then $|y-x|\sim|y-x_0|$. Since $s/p>1$ and $w^{s/p}\in A_1$, then by Lemma C, we have $w\in A_1\cap RH_{s/p}$. Consequently, it follows from the inequality (5.3), H\"older's inequality and the $A_p$ condition that
\begin{equation*}
\begin{split}
M_{\beta,1}(f_2)(x)\le&\,C\int_{(2B)^c}\frac{|f(y)|}{|x-y|^{n-\beta}}\,dy\\
\le&\, C\sum_{j=1}^\infty\frac{1}{|2^{j+1}B|^{1-\beta/n}}\int_{2^{j+1}B}|f(y)|\,dy\\
\le&\, C\sum_{j=1}^\infty\frac{1}{|2^{j+1}B|^{1-\beta/n}}\cdot|2^{j+1}B|w(2^{j+1}B)^{-1/p}
\bigg(\int_{2^{j+1}B}|f(y)|^pw(y)\,dy\bigg)^{1/p}\\
\le&\, C\|f\|_{L^{p,\kappa}(w)}\sum_{j=1}^\infty|2^{j+1}B|^{\beta/n}w(2^{j+1}B)^{(\kappa-1)/p}.
\end{split}
\end{equation*}
Hence
\begin{equation*}
\begin{split}
K_4&\le C\|f\|_{L^{p,\kappa}(w)}\cdot\frac{w^{s/p}(B)^{1/s}}{w(B)^{\kappa/p}}\sum_{j=1}^\infty|2^{j+1}B|^{\beta/n}w(2^{j+1}B)^{(\kappa-1)/p}\\
&\le C\|f\|_{L^{p,\kappa}(w)}\cdot\frac{|B|^{-\beta/n}w(B)^{1/p}}{w(B)^{\kappa/p}}\sum_{j=1}^\infty|2^{j+1}B|^{\beta/n}w(2^{j+1}B)^{(\kappa-1)/p}\\
&=C\|f\|_{L^{p,\kappa}(w)}\sum_{j=1}^\infty\frac{|2^{j+1}B|^{\beta/n}}{|B|^{\beta/n}}\cdot\frac{w(B)^{(1-\kappa)/p}}{w(2^{j+1}B)^{(1-\kappa)/p}}.
\end{split}
\end{equation*}
Since $r_w>\frac{1-\kappa}{p/s-\kappa}$, then we can find a suitable number $r$ such that $r>\frac{1-\kappa}{p/s-\kappa}$ and $w\in RH_r$. Furthermore, by using Lemma B, we get
$$\frac{w(B)}{w(2^{j+1}B)}\le C\bigg(\frac{|B|}{|2^{j+1}B|}\bigg)^{(r-1)/r}.$$
Therefore
\begin{align}
K_4&\le C\|f\|_{L^{p,\kappa}(w)}\sum_{j=1}^\infty\big(2^{jn}\big)^{\beta/n-(r-1)(1-\kappa)/{pr}}\notag\\
&\le C\|f\|_{L^{p,\kappa}(w)},
\end{align}
where the last series is convergent since $\beta/n-(r-1)(1-\kappa)/{pr}<0$. Combining the above inequality (5.5) with (5.4) and taking the supremum over all balls $B\subseteq \mathbb R^n$, we get the desired result.
\end{proof}

\begin{prop}
Let $0<\delta<1$, $w\in A_1$, $0<\beta<1$ and $b\in Lip_{\beta}(w)$. Then for all $r>1$ and for all $x\in \mathbb R^n$, we have
\begin{equation*}
\begin{split}
M^{\#}_\delta([b,T]f)(x)\le& C\|b\|_{Lip_\beta(w)}\Big(w(x)M_{\beta,r,w}(Tf)(x)+w(x)M_{\beta,r,w}(f)(x)\\
&+w(x)^{1+\beta/n}M_{\beta,1}(f)(x)\Big).
\end{split}
\end{equation*}
\end{prop}
\begin{proof}
Again, as in the proof of Proposition 3.3, we will split the previous expression (3.1) into three parts and estimate each term separately. Since $w\in A_1$, then it follows from H\"older's inequality and Lemma D that
\begin{align}
\mbox{\upshape I}&\le C\cdot\frac{1}{|B|}\int_B\big|(b(y)-b_{2B})Tf(y)\big|\,dy\notag\\
&\le C\cdot\frac{1}{|B|}\bigg(\int_B\big|b(y)-b_{2B}\big|^{r'}w^{1-r'}\,dy\bigg)^{1/{r'}}
\bigg(\int_B\big|Tf(y)\big|^rw(y)\,dy\bigg)^{1/r}\notag\\
&\le C\|b\|_{Lip_\beta(w)}\cdot\frac{w(B)}{|B|}\bigg(\frac{1}{w(B)^{1-{\beta r}/n}}\int_B\big|Tf(y)\big|^rw(y)\,dy\bigg)^{1/r}\notag\\
&\le C\|b\|_{Lip_\beta(w)}w(x)M_{\beta,r,w}(Tf)(x).
\end{align}
As before, by Kolmogorov's inequality, H\"older's inequality and Lemma D, we thus obtain
\begin{align}
\mbox{\upshape II}&\le C\cdot\frac{1}{|B|}\int_{2B}\big|(b(y)-b_{2B})f(y)\big|\,dy\notag\\
&\le C\|b\|_{Lip_\beta(w)}w(x)M_{\beta,r,w}(f)(x).
\end{align}
Following along the same lines as that of Proposition 3.3, we have
$$\mbox{\upshape III}\le \mbox{\upshape{IV+V}},$$
where
$$\mbox{\upshape{IV}}=C\sum_{j=1}^\infty\frac{1}{2^j}\frac{1}{|2^{j+1}B|}
\int_{2^{j+1}B}|b(z)-b_{2^{j+1}B}||f(z)|\,dz$$
and
$$\mbox{\upshape{V}}=C\sum_{j=1}^\infty\frac{1}{2^j}\frac{1}{|2^{j+1}B|}
\int_{2^{j+1}B}|b_{2^{j+1}B}-b_{2B}||f(z)|\,dz.$$
Similarly, we can get
\begin{equation}
\mbox{\upshape{IV}}\le C\|b\|_{Lip_\beta(w)}w(x)M_{\beta,r,w}(f)(x).
\end{equation}
Observe that $w\in A_1$, then by Lemma D, a simple calculation gives that
\begin{equation*}
|b_{2^{j+1}B}-b_{2B}|\le C\|b\|_{Lip_\beta(w)}j\cdot w(x)w(2^{j+1}B)^{\beta/n}.
\end{equation*}
Therefore
\begin{align}
\mbox{\upshape{V}}&\le C\|b\|_{Lip_\beta(w)}\sum_{j=1}^\infty\frac{j}{2^j}\cdot\frac{w(x)w(2^{j+1}B)^{\beta/n}}{|2^{j+1}B|}
\int_{2^{j+1}B}|f(z)|\,dz\notag\\
&\le C\|b\|_{Lip_\beta(w)}\sum_{j=1}^\infty\frac{j}{2^j}\cdot w(x)^{1+\beta/n}\frac{1}{|2^{j+1}B|^{1-\beta/n}}\int_{2^{j+1}B}|f(z)|\,dz\notag\\
&\le C\|b\|_{Lip_\beta(w)}w(x)^{1+\beta/n}M_{\beta,1}(f)(x)\sum_{j=1}^\infty\frac{j}{2^j}\notag\\
&\le C\|b\|_{Lip_\beta(w)}w(x)^{1+\beta/n}M_{\beta,1}(f)(x).
\end{align}
Summarizing the above estimates (5.6)--(5.9) and taking the supremum over all balls $B\subseteq\mathbb R^n$, we obtain the desired result.
\end{proof}

Finally let us give the proof of Theorem 1.3.

\begin{proof}[Proof of Theorem $1.3$]
For $0<\beta<1$ and $1<p<n/\beta$, we are able to choose a suitable number $r$ such that $1<r<p$. By Proposition 3.1 and Proposition 5.4, we have
\begin{equation*}
\begin{split}
&\big\|[b,T]f\big\|_{L^{s,{\kappa s}/p}(w^{1-s},w)}\\
\le& \,C\big\|M^{\#}_\delta([b,T]f)\big\|_{L^{s,{\kappa s}/p}(w^{1-s},w)}\\
\le& \,C\|b\|_{Lip_\beta(w)}\Big(\|w(\cdot)M_{\beta,r,w}(Tf)\|_{L^{s,{\kappa s}/p}(w^{1-s},w)}\\
&+\|w(\cdot)M_{\beta,r,w}(f)\|_{L^{s,{\kappa s}/p}(w^{1-s},w)}+\|w(\cdot)^{1+\beta/n}M_{\beta,1}(f)\|_{L^{s,{\kappa s}/p}(w^{1-s},w)}\Big)\\
\le& \,C\|b\|_{Lip_\beta(w)}\Big(\|M_{\beta,r,w}(Tf)\|_{L^{s,{\kappa s}/p}(w)}\\
&+\|M_{\beta,r,w}(f)\|_{L^{s,{\kappa s}/p}(w)}+\|M_{\beta,1}(f)\|_{L^{s,{\kappa s}/p}(w^{s/p},w)}\Big).
\end{split}
\end{equation*}
Applying Theorem G, Lemma 5.2 and Lemma 5.3, we finally obtain
\begin{equation*}
\begin{split}
\big\|[b,T]f\big\|_{L^{s,{\kappa s}/p}(w^{1-s},w)}&\le C\|b\|_{Lip_\beta(w)}\Big(\|Tf\|_{L^{p,\kappa}(w)}+\|f\|_{L^{p,\kappa}(w)}\Big)\\
&\le C\|b\|_{Lip_\beta(w)}\|f\|_{L^{p,\kappa}(w)}.
\end{split}
\end{equation*}
Therefore, we conclude the proof of Theorem 1.3.
\end{proof}

\end{document}